\documentclass{amsart} 
\usepackage{amsmath}
\usepackage[mathscr]{eucal} 
\usepackage{amssymb}
\usepackage{latexsym}
\usepackage{amsthm} 
\theoremstyle{plain}
\newtheorem{theorem}{Theorem}[section]

\newtheorem{theorema}{Theorem A\!\!}
\newtheorem{theoremb}{Theorem B\!\!}

\newtheorem{lemma}[theorem]{Lemma}

\newcommand{\diam}{\mathop{\mathrm{diam}}\nolimits} 
\numberwithin{equation}{section}  
\theoremstyle{definition}

\theoremstyle{remark}
\newtheorem{remark}[theorem]{Remark}

\def\XXint#1#2#3{{\setbox0=\hbox{$#1{#2#3}{\int}$}
\vcenter{\hbox{$#2#3$}}\kern-.5\wd0}}

\title[Decomposition of neighborhood of a circular cone]
{On decomposition of  neighborhood of a circular cone related to principal 
curvatures}  
\author{Shuichi Sato} 
 
\begin{document} 
\address{Department of Mathematics,
Faculty of Education,
Kanazawa University,
Kanazawa 920-1192,
Japan}
\email{shuichi@kenroku.kanazawa-u.ac.jp, shuichipm@gmail.com}

\begin{abstract} 
We give an alternative proof of a result on the uniform overlap of 
the algebraic sums of the  sets arising from a decomposition of 
a neighborhood of a circular cone in 
$\Bbb R^3$.  It is known that the uniform overlap result can be applied to 
make a unified approach for the proofs of a theorem on the maximal 
Bochner-Riesz operator on $\Bbb R^2$ and a theorem on the 
maximal spherical means on $\Bbb R^2$.  
\end{abstract}
  \thanks{ 
  \endgraf
  }
\thanks{The author is partly supported
by Grant-in-Aid for Scientific Research (C) No. 20K03651, Japan Society for the  Promotion of Science.}

\maketitle 

\section{Introduction}\label{sec1}  
Let 
\begin{equation*} 
T^\lambda_Rf(x)=\int_{|\xi|<R}\hat{f}(\xi)(1-|R^{-1}\xi|^2)_+^\lambda 
e^{2\pi i\langle x, \xi\rangle}\, d\xi    
\end{equation*} 
be the Bochner-Riesz operator of order $\lambda$ on $\Bbb R^2$,
 where 
\begin{equation*} 
\hat{f}(\xi)= \int_{\Bbb R^2} f(x)e^{-2\pi i\langle x, \xi\rangle}\, dx
\end{equation*} 
is the Fourier  transform and $\langle x, \xi\rangle=x_1\xi_1+x_2\xi_2$, 
$x=(x_1,x_2)$, $\xi=(\xi_1, \xi_2)$,  
denotes the inner product.  Let 
\begin{equation*} 
T_*^\lambda f(x)=\sup_{R>0}|T_R^\lambda f(x)|
\end{equation*} 
be the maximal Bochner-Riesz operator.    
\par 
The following is known (\cite{Ca}). 
\begin{theorema}
  If $\lambda>0$, $T_*^\lambda$ is bounded on $L^4(\Bbb R^2):$  
\begin{equation*} 
\|T_*^\lambda f\|_4\leq C_\lambda \|f\|_4. 
\end{equation*} 
\end{theorema} 
The $L^4$ boundedness for $T_1^\lambda$ is shown in \cite{Ca-Sj}. 
See also \cite{Ca2} for related results. 
\par 
Let 
\begin{equation*} 
S_tf(x)=\int_{S^{1}} f(x-t\theta)\, d\sigma(\theta)
\end{equation*} 
be the spherical mean on $\Bbb R^2$, where $S^1=\{x\in \Bbb R^2: |x|=1\}$ 
is the unit circle 
and $\sigma$ denotes the Lebesgue arc length measure on $S^1$, and let 
\begin{equation*} 
S_*f(x)=\sup_{t>0} |S_tf(x)| 
\end{equation*} 
be the maximal spherical mean.  
\par 
The following result is known (\cite{Bo}). 
\begin{theoremb}  
The maximal operator $S_*$ is bounded on $L^p(\Bbb R^2)$ for $p>2:$  
\begin{equation*} 
\|S_*f\|_p\leq C_p \|f\|_p. 
\end{equation*} 
\end{theoremb} 
We refer to \cite{SW} for a result analogous to Theorem B in 
$\Bbb R^n$,  $n\geq 3$.   
\par 
In \cite[Chap. 2]{Sogg}, a unified approach to the proofs of Theorems A and B 
are  presented. In the arguments, a geometric overlap theorem concerning 
a circular cone in $\Bbb R^3$ plays a crucial role. 
\par   
Let $\tau$ be a fixed large positive number. In this note we assume that 
$\tau>10^6$.  Set 
\begin{equation}\label{e.1.1}  
\Gamma_\mu=\{\xi \in \Bbb R^2\setminus \{0\}: 
\mu\tau^{-1/2}\leq \arg \xi <(\mu+1)\tau^{-1/2}\}, 
\end{equation}  
where $\mu\in \Bbb R$ with $|\mu|\leq \tau^{1/2}(\pi/8)-1$. Let 
\begin{equation}\label{e.1.2}
 \mathcal N= \{\mu\in \Bbb Z: |\mu|\leq \tau^{1/2}(\pi/8)-1\}, 
\end{equation}  
where $\Bbb Z$ denotes the set of integers. 
Let $J=[\alpha, \beta] \subset [1, 2]$. We assume that $|J|=\beta-\alpha 
\leq \tau^{-1/2}$.  
For $\mu\in \mathcal N$,   let 
\begin{gather}\label{e.1.3}  
U_{\mu,J}=\{(\xi,|\xi|)\in \Bbb R^2\times \Bbb R: \xi\in \Gamma_\mu, 
|\xi|\in J\} 
\end{gather} 
and 
\begin{align}\label{e.1.4}  
u_{\mu,J}&=
\{(\xi,\eta)\in \Bbb R^2\times \Bbb R: |\eta-|\xi||\leq \tau^{-1}, 
(\xi,|\xi|)\in U_{\mu,J}\}  
\\ 
&= \cup \{\{\xi\}\times [|\xi|-\tau^{-1}, |\xi|+\tau^{-1}]: 
(\xi,|\xi|)\in  U_{\mu,J}\}.                                  \notag 
\end{align} 
We note that $U_{\mu,J}\subset u_{\mu,J}$. 
\par 
We have the following result. 

\begin{theorem} \label{th1.1} 
Let $[\alpha_1, \beta_1], [\alpha_2, \beta_2] \subset [1,2]$. Set 
$J_i=[\alpha_i, \beta_i]$ and suppose that  
$|J_i|\leq \tau^{-1/2},$ $i=1, 2$. 
Let $u_{\mu,J_1}$, $u_{\mu,J_2}$ be defined as in \eqref{e.1.4} with 
$J_1, J_2$ in place of $J$.   
Then there exists a constant $C$ independent of $\tau$ and the intervals 
$J_1, J_2$ such that 
\begin{equation*}  
I=\sum_{(\mu, \nu)\in \mathcal N^2} \chi_{u_{\mu,J_1}+u_{\nu,J_2}} \leq C, 
\end{equation*} 
where $\mathcal N$ is as in \eqref{e.1.2},  
$u_{\mu,J_1}+u_{\nu,J_2}=\{(\xi,\eta)+(\xi',\eta'): 
(\xi,\eta)\in u_{\mu,J_1}, (\xi',\eta')\in u_{\nu,J_2}\}$ and $\chi_E$ denotes 
the characteristic function of a set $E$.   
\end{theorem}  
\par 
A more general result is stated in \cite[(2.4.23), p. 87]{Sogg}, 
 and a variant of it is also given (\cite[Lemma 2.4.5]{Sogg}).  
As in \cite{Sogg}, in the arguments for the unified approach to the proofs of 
Theorems A and B applying half wave operators, 
Theorem \ref{th1.1} can be used to prove an orthogonality 
result through Fourier transform estimates, which is crucial in the arguments, 
since the orthogonality result leads to an effective application of sharp
 $L^2(\Bbb R^3)$-$L^2(\Bbb R^2)$ estimates 
for the Kakeya maximal functions defined by using rays on a light cone.  
\par 
The result \cite[(2.4.23)]{Sogg} includes Theorem \ref{th1.1} as a special 
case and the proof is given in \cite[pp. 87-88]{Sogg}. 
In this note we shall give an alternative proof of Theorem \ref{th1.1}. 
Also, we shall consider a 
variant of Theorem \ref{th1.1} (Theorem \ref{th4.2} below),  
which is a special case of \cite[Lemma 2.4.5]{Sogg} and 
which is related to Theorem \ref{th1.1} as 
\cite[Lemma 2.4.5]{Sogg} is related to \cite[(2.4.23)]{Sogg}. 
We shall give the proof of Theorem \ref{th4.2} by applying 
Theorem \ref{th1.1}. 
\par  
A version of Theorem \ref{th1.1} on $\Bbb R^2$, where a cone is replaced with 
a circle, is given in \cite{Fe} and it follows from Theorem \ref{th1.1} as 
a corollary. 
\par 
\begin{remark} 
If the condition on $\mu$ in the definition of $\mathcal N$ in \eqref{e.1.2} is  replaced by $|\mu|\leq \tau^{1/2}a-1$ with some $\pi/2<a<\pi$, then 
a result analogous to Theorem \ref{th1.1} will not exist.  This can be seen by letting $J_1=J_2=:J$, $b\in J$  and considering $(\mu,\nu)$ satisfying 
$(\xi, |\xi|)\in U_{\mu,J}$ and $(-\xi, |\xi|)\in U_{\nu,J}$ for some $\xi$ 
such that $|\xi|=b$ and $\arg \xi$ is sufficiently close to $\pi/2$.  
For such $(\mu,\nu)$ we have $(0,2b)\in 
U_{\mu,J}+U_{\nu,J}$ and we note that the cardinality of the family of such 
$(\mu,\nu)$ increases with $\tau$ unlimitedly.   
\end{remark}  
Let 
\begin{equation}\label{e.1.5} 
\mathcal I=\left\{a: a=\frac{\mu+\nu}{2}, \mu, \nu\in \mathcal N \right\}.   
\end{equation} 
Set $\Bbb Z^*=\{k/2: k\in \Bbb Z\}$. Then we note that $\mathcal I$ is 
a subset of $\Bbb Z^*$ and if $a\in \mathcal I$, then 
$-(\pi/8)\tau^{1/2} +1\leq a\leq (\pi/8)\tau^{1/2} -1$.  
To prove Theorem \ref{th1.1}, we write 
\begin{equation*}  
I=\sum_{(\mu, \nu)\in \mathcal N^2} \chi_{u_{\mu,J_1}+u_{\nu,J_2}} = 
\sum_{a\in \mathcal I} \sum_{(\mu, \nu)\in k_0^{-1}(\{a\})} 
\chi_{u_{\mu,J_1}+u_{\nu,J_2}},   
\end{equation*}  
where the surjection $k_0: \mathcal N^2 \to 
\mathcal I$ is defined by 
\begin{equation*} 
k_0(\mu,\nu)=\frac{\mu+\nu}{2}.  
\end{equation*}  
Let $\mathcal I'\subset \mathcal I$ be such that if $a, a' \in \mathcal I'$ and  $a\neq a'$, then $|a-a'|\geq C_0$, where $C_0$ is sufficiently large. 
To prove $I\leq C$, it suffices to show that 
\begin{equation}  \label{e.1.6}
\sum_{a\in \mathcal I'} \sum_{(\mu, \nu)\in k_0^{-1}(\{a\})} 
\chi_{u_{\mu,J_1}+ u_{\nu,J_2}} \leq C',  
\end{equation}  
by considering a suitable partition of $\mathcal I$. 
\par 
In the proof of \eqref{e.1.6}, one result we apply is the following. 
\begin{lemma} \label{lem1.2}
Fix $a \in \mathcal I$. Let $\frac{\mu+\nu}{2}=a$, $\mu, \nu\in \mathcal N$
 and 
\begin{equation*}  
E_a(\mu,\nu)=u_{\mu,J_1}+ u_{\nu,J_2}.  
\end{equation*}  
Then $\{E_a(\mu,\nu)\}_{(\mu, \nu)\in k_0^{-1}(\{a\})}$ is finitely 
overlapping uniformly in $a\in \mathcal I$.    
\end{lemma}
\par 
Let 
\begin{equation*}\label{}  
\widetilde{J_1}= [\alpha_1-\tau^{-1}, \beta_1+\tau^{-1}], \quad 
\widetilde{J_2}= [\alpha_2-\tau^{-1}, \beta_2+\tau^{-1}]. 
\end{equation*}
Let $\eta\in \widetilde{J_1} +\widetilde{J_2}= 
 [\alpha_1+\alpha_2-2\tau^{-1}, \beta_1+\beta_2+2\tau^{-1}]$ 
and define 
\begin{equation*}\label{}  
E^\eta_a(\mu,\nu)=\left\{\xi \in \Bbb R^2: (\xi,\eta) \in E_a(\mu,\nu)
\right\}, \quad (\mu, \nu)\in k_0^{-1}(\{a\}). 
\end{equation*}
Then the following result implies Lemma \ref{lem1.2}. 
\begin{lemma} \label{lem1.3}
Fix $a \in \mathcal I$. Let $\eta\in \widetilde{J_1} +\widetilde{J_2}$. 
Then $\{E_a^\eta(\mu,\nu)\}_{(\mu, \nu)\in k_0^{-1}(\{a\})}$ is finitely 
overlapping uniformly in $a$ and $\eta$. 
\end{lemma}  
To prove Lemma \ref{lem1.3}, we observe that 
\begin{equation*}\label{}  
E^\eta_a(\mu,\nu)=\bigcup\left\{u_{\mu,J_1}^{\eta'}+u_{\nu,J_2}^{\eta''}: 
\eta'+\eta''=\eta, \eta'\in \widetilde{J_1},  
\eta''\in \widetilde{J_2}\right\}, 
\end{equation*}
where $u_{\mu,J_1}^{\eta'}$ and $u_{\nu,J_2}^{\eta''}$ are defined from 
$u_{\mu,J_1}$ and $u_{\nu,J_2}$ as 
\begin{equation*} 
u_{\mu,J_1}^{\eta'}= \left\{\xi \in \Bbb R^2: (\xi,\eta') 
\in u_{\mu,J_1}\right\}, 
\quad u_{\nu,J_2}^{\eta''}= 
\left\{\xi \in \Bbb R^2: (\xi,\eta'') \in u_{\nu,J_2}\right\}. 
\end{equation*}
Thus Lemma \ref{lem1.3} is restated as follows. 
\begin{lemma}\label{lem1.4} Let  $a \in \mathcal I$ and 
$\eta\in \widetilde{J_1} +\widetilde{J_2}$. The family of the sets 
\begin{equation*} 
\left\{ \bigcup_{\substack{\eta'+\eta''=\eta, \\ 
\eta'\in \widetilde{J_1}, 
\eta''\in \widetilde{J_2}}}\left(u_{\mu,J_1}^{\eta'}+u_{\nu,J_2}^{\eta''} 
\right)\right\}_{(\mu, \nu)\in k_0^{-1}(\{a\})} 
\end{equation*} 
is finitely overlapping uniformly in $a \in \mathcal I$ and 
$\eta\in \widetilde{J_1} +\widetilde{J_2}$. 
\end{lemma}
\par 
To prove \eqref{e.1.6} another result we need is the following.  
\begin{lemma}\label{lem1.5} 
For  $a \in \mathcal I$,  let 
\begin{equation*} 
E_a=\bigcup_{(\mu, \nu)\in k_0^{-1}(\{a\})} E_a(\mu,\nu). 
\end{equation*} 
Then there exists $C_0>0$ such that if $|a-a'|\geq C_0$, 
$a, a'\in \mathcal I$, then $E_a\cap E_{a'}=\emptyset$.    
\end{lemma}

Lemma \ref{lem1.5} follows from the following. 

\begin{lemma}\label{lem1.6} 
Let $\eta\in \widetilde{J_1} +\widetilde{J_2}$. 
Then there exists $C_0>0$ independent of  
$\eta$ such that $E_a^\eta\cap E_{a'}^\eta =\emptyset$ if 
$|a-a'|\geq C_0$, $a, a' \in \mathcal I$,    where  
\begin{equation*} 
E_a^\eta=\{\xi \in \Bbb R^2: (\xi, \eta)\in E_a\}. 
\end{equation*} 
\end{lemma}
\par 
By Lemmas \ref{lem1.2} and \ref{lem1.5} we have \eqref{e.1.6}, from which 
Theorem \ref{th1.1} will follow.  As we have seen above, 
Lemmas \ref{lem1.2} and \ref{lem1.5} follow from Lemmas \ref{lem1.4} 
and \ref{lem1.6}, respectively. 
So, to prove Theorem \ref{th1.1} it suffices to show Lemmas \ref{lem1.4} 
and \ref{lem1.6}.  
\par 
In Section \ref{s2}, we shall prove Lemma \ref{lem1.4} by applying 
arguments using principal curvatures of a circular cone. 
The proof of Lemma \ref{lem1.6} will be given in Section \ref{s3}. 
When $J_1=J_2$, we can prove Lemma \ref{lem1.6} by observing that $E_a^\eta$ 
is contained in a $c\tau^{-1/2}$ neighborhood $\widetilde{\ell_a}$ of a line 
segment $\ell_a$ for some positive constant $c$, where   
\begin{equation*} 
\ell_a=\{\xi\in \Bbb R^2: \arg \xi=a\tau^{-1/2}, 1/2\leq |\xi|\leq 9/2\},  
\quad \widetilde{\ell_a}=
\{\zeta\in \Bbb R^2: d(\zeta, \ell_a)< c\tau^{-1/2}\},   
\end{equation*} 
with $d(\zeta, \ell_a)=\inf_{\xi\in \ell_a}|\zeta -\xi|$.  
The proof for the general case is slightly less straightforward.  
We shall provide a detailed proof. 
In Section \ref{s4}, we shall state a variant of Theorem \ref{th1.1} 
(Theorem \ref{th4.2}) and give the proof.

\section{Proof of Lemma $\ref{lem1.4}$} \label{s2} 

We need the following. 

\begin{lemma}\label{lem2.1}  
Let $\eta'\in \widetilde{J_1}$, $\mu \in \mathcal N$. 
\begin{enumerate} 
\item[(1)] 
if $\xi \in u^{\eta'}_{\mu,J_1}$, then    
\begin{equation*} 
\xi=\eta'(\cos \theta, \sin\theta) +\zeta, \quad 
\zeta=\sigma(\cos \theta, \sin\theta)   
\end{equation*} 
for some $\theta, \sigma\in \Bbb R$ such that 
$\mu\tau^{-1/2}\leq \theta<(\mu+1)\tau^{-1/2}$ and $|\sigma|\leq \tau^{-1}$;  
\item[(2)] 
if $\mu\tau^{-1/2}\leq \theta<(\mu+1)\tau^{-1/2}$, 
there exists $\sigma\in \Bbb R$ such that $|\sigma|\leq \tau^{-1}$ and 
\begin{equation*} 
\eta'(\cos \theta, \sin\theta)+ \sigma(\cos \theta, \sin\theta) \in 
u^{\eta'}_{\mu,J_1}. 
\end{equation*} 
\end{enumerate}  
Similar results  hold for $ u^{\eta''}_{\nu,J_2}$ with 
$\eta''\in \widetilde{J_2}$, $\nu \in \mathcal N$.
\end{lemma} 
\begin{proof} 
If $\xi \in u^{\eta'}_{\mu,J_1}$, then $(\xi, \eta')\in u_{\mu,J_1}$, 
which implies that 
$(\xi,|\xi|)\in U_{\mu,J_1}$ and $|\eta'-|\xi||\leq \tau^{-1}$.  
Since $(\xi,|\xi|)\in U_{\mu,J_1}$, there exists $\theta\in \Bbb R$ such that 
$\mu\tau^{-1/2}\leq \theta<(\mu+1)\tau^{-1/2}$ and 
$\xi=|\xi|(\cos \theta, \sin\theta)$. We write 
\begin{equation*} 
\xi=\eta'(\cos \theta, \sin\theta)+ (|\xi|-\eta')(\cos \theta, \sin\theta). 
\end{equation*} 
Putting $\sigma=|\xi|-\eta'$, we get the conclusion of part (1). 
\par 
Proof of part (2). 
We take $\eta_0' \in J_1$ such that $|\eta_0'-\eta'|\leq \tau^{-1}$. Then 
$(\eta_0'(\cos\theta, \sin\theta), \eta_0')\in U_{\mu,J_1}$. Thus 
 $(\eta_0'(\cos\theta, \sin\theta), \eta')\in u_{\mu,J_1}$. It follows that 
$\eta_0'(\cos\theta, \sin\theta)\in u^{\eta'}_{\mu,J_1}$.  Therefore, setting 
$\sigma=\eta_0'-\eta'$, we reach the conclusion.   
\end{proof}

\begin{proof}[Proof of Lemma $\ref{lem1.4}$]  
We first assume that $a=0$.  Let $(\mu,\nu)\in k_0^{-1}(\{0\})$, $\eta'\in 
\widetilde{J_1}$,  $\eta''\in \widetilde{J_2}$.  
Suppose that $\mu=\ell +m$, $\nu=-\ell - m$, 
with $\ell, m\geq 0$. By Lemma \ref{lem2.1} (2), there exist  
$p\in u^{\eta'}_{\mu,J_1}$ and $q\in u^{\eta''}_{\nu,J_2}$ such that 
\begin{gather} \label{e.0.2.1}
|\eta'(\cos(\mu\tau^{-1/2}), \sin(\mu\tau^{-1/2})) -p|\leq \tau^{-1}, 
\\ \label{e.0.2.2} 
|\eta''(\cos(\nu\tau^{-1/2}), \sin(\nu\tau^{-1/2})) -q|\leq \tau^{-1}. 
\end{gather} 
Also, we have 
\begin{align}\label{e.0.2.3}  
\eta'(\cos(\ell+m)\tau^{-1/2}, \sin(\ell+m)\tau^{-1/2}) + 
\eta''(\cos(\ell+m)\tau^{-1/2}, -\sin(\ell+m)\tau^{-1/2}) 
\\ 
= (\eta\cos(\ell+m)\tau^{-1/2}, (\eta'-\eta^{\prime\prime})\sin(\ell+m)\tau^{-1/2}).                                       \notag 
\end{align}  
We note that 
\begin{equation*} 
\cos \ell\tau^{-1/2} - \cos(\ell+m)\tau^{-1/2}=
2\sin\left(\ell+\frac{m}{2}\right)\tau^{-1/2}\sin\frac{m}{2}\tau^{-1/2}.  
\end{equation*}
By this it follows that 
\begin{align} \label{e.2.1}  
\left|\cos \ell\tau^{-1/2} - \cos(\ell+m)\tau^{-1/2} \right| 
&\leq 2\left(\ell+\frac{m}{2}\right)\frac{m}{2}\tau^{-1} 
\\ 
 \label{e.2.2}  
\left|\cos \ell\tau^{-1/2} - \cos(\ell+m)\tau^{-1/2} \right| 
&\geq 2(2/\pi)^2 \left(\ell+\frac{m}{2}\right)\frac{m}{2}\tau^{-1}, 
\end{align} 
where we have used well-known inequalities $\sin x\leq x$, $x\geq 0$, and 
$\sin x\geq (2/\pi)x$,  $0\leq x\leq \pi/2$.  
\par 
If $\xi \in u^{\eta'}_{\mu,J_1}$, $\xi=(\xi_1, \xi_2)$, $\mu=\ell+m$, 
by Lemma \ref{lem2.1} (1) and the estimate $|\eta'|\leq 3$ and 
by using \eqref{e.2.1} suitably,  we have 
\begin{align*} 
|\eta'\cos(\ell+m)\tau^{-1/2}-\xi_1|&\leq |\eta'\cos(\ell+m)\tau^{-1/2}-
\eta'\cos\theta| +|\zeta_1| 
\\ 
&\leq \eta'|\cos(\ell+m)\tau^{-1/2}-\cos(\ell+m+1)\tau^{-1/2}| +\tau^{-1} 
\\ 
&\leq 3(\ell+ m+1)\tau^{-1}.    
\end{align*}
Also,  if $\xi' \in u^{\eta''}_{\nu,J_2}$, $\xi'=(\xi'_1,\xi'_2)$,  
$\nu=-(\ell+m)$, 
\begin{equation*} 
\xi'=\eta''(\cos \theta', \sin\theta') +\zeta', \quad 
\zeta'=\sigma'(\cos \theta', \sin\theta'),  
\end{equation*}
with $\nu\tau^{-1/2}\leq \theta'<(\nu+1)\tau^{-1/2}$, 
$|\sigma'|\leq \tau^{-1}$, then  
\begin{align*} 
|\eta''\cos(\ell+m)\tau^{-1/2}-\xi'_1|&\leq \eta''|\cos(\ell+m)\tau^{-1/2}-
\cos(\ell+m-1)\tau^{-1/2}| +|\zeta'_1| 
\\ 
&\leq \eta''|\ell+m-1/2|\tau^{-1} +\tau^{-1} 
\\ 
&\leq 3(\ell+ m+1)\tau^{-1}    
\end{align*} 
for $\ell, m\geq 0$.  
Thus we have 
$\diam P_1(u^{\eta'}_{\mu,J_1})\leq 6(\ell+m+1)\tau^{-1}$, 
$\diam P_1(u^{\eta''}_{\nu,J_2})\leq 6(\ell+m+1)\tau^{-1}$, where 
$P_1$ is the projection mapping defined by $P_1(\xi)=\xi_1$ when 
$\xi=(\xi_1,\xi_2)$. 
\par 
Therefore 
\begin{equation} \label{e.2.3} 
\diam P_1(u^{\eta'}_{\mu,J_1}+ u^{\eta''}_{\nu,J_2}) 
\leq 12(\ell+m+1)\tau^{-1}. 
\end{equation} 
Let $\widetilde{\eta}'+\widetilde{\eta}''=\eta, 
\widetilde{\eta}'\in \widetilde{J_1}, 
\widetilde{\eta}''\in \widetilde{J_2}$.   
By \eqref{e.0.2.1}, \eqref{e.0.2.2} and \eqref{e.0.2.3}, 
there exist $A\in u^{\eta'}_{\ell,J_1}+ u^{\eta''}_{-\ell,J_2}$ and 
$B\in u^{\widetilde{\eta}'}_{\ell+m,J_1}+ 
u^{\widetilde{\eta}''}_{-\ell-m,J_2}$ 
such that 
\begin{equation*} 
|\eta\cos\ell\tau^{-1/2}- P_1(A)|\leq 2\tau^{-1}, \quad 
|\eta\cos(\ell+m)\tau^{-1/2} -P_1(B)|\leq 2\tau^{-1}. 
\end{equation*} 
Thus if 
\begin{equation*} 
P_1(u^{\eta'}_{\ell,J_1}+ u^{\eta''}_{-\ell,J_2}) \cap 
P_1(u^{\widetilde{\eta}'}_{\ell+m,J_1}+ u^{\widetilde{\eta}''}_{-\ell-m,J_2})
\neq \emptyset, 
\end{equation*} 
then 
\begin{align*} 
12(\ell+1)\tau^{-1}+12(\ell+m+1)\tau^{-1}&\geq |P_1(A)-P_1(B)| 
\\ 
&\geq \eta\cos\ell\tau^{-1/2} 
- \eta\cos(\ell+m)\tau^{-1/2} -4\tau^{-1} 
\\ 
&\geq \eta 2\pi^{-2}(2\ell+m)m\tau^{-1}  -4\tau^{-1} 
\\ 
&\geq 2\pi^{-2}(2\ell+m)m\tau^{-1}  -4\tau^{-1}, 
\end{align*} 
where the penultimate inequality follows by \eqref{e.2.2}.  This implies that 
\begin{equation*} 
m^2+2(\ell-3\pi^2)m-2\pi^2(6\ell+7)\leq 0, 
\end{equation*} 
and hence we see that $m\leq C$ with a positive constant $C$. 
From this Lemma \ref{lem1.4} for $a=0$ can be deduced.  
\par 
Let $\mathcal R_\sigma$ be a rotation around the origin such that 
$\mathcal R_\sigma((1,0))=(\cos\sigma, \sin\sigma)$.  
To prove the general case, let $a\in \mathcal I$, 
$(\mu, \nu)\in k_0^{-1}(\{a\})$ and put  
 $\alpha=\mu-a$, $\beta=\nu-a$. Then $\alpha+\beta=0$. 
We note that $\alpha, \beta\in \Bbb Z^*
\cap [-\tau^{1/2}\pi 8^{-1}+1, \tau^{1/2}\pi 8^{-1}-1]$, recalling 
$\Bbb Z^*=\{k/2: k\in \Bbb Z\}$.  
Also, we observe that 
$\mathcal R_{-a\tau^{-1/2}}\Gamma_\mu=\Gamma_\alpha$ and 
$\mathcal R_{-a\tau^{-1/2}}\Gamma_{\nu}=\Gamma_{\beta}$.  Thus, we can 
argue similarly to the case $a=0$ to handle the family of the sets 
\begin{equation*} 
\left\{\bigcup_{\substack{\eta'+\eta''=\eta, \\ 
\eta'\in \widetilde{J_1}, 
\eta''\in \widetilde{J_2}}}\mathcal R_{-a\tau^{-1/2}} 
\left(u_{\mu,J_1}^{\eta'}+u_{\nu,J_2}^{\eta''}\right)
\right\}_{(\mu, \nu)\in k_0^{-1}(\{a\})}
\end{equation*} 
 to get a finitely 
overlapping property which can prove the desired result by applying 
the mapping $\mathcal R_{a\tau^{-1/2}}$. 
This completes the proof of Lemma \ref{lem1.4}. 
\end{proof}

\section{Proof of Lemma $\ref{lem1.6}$} \label{s3}

In this section, we prove Lemma \ref{lem1.6}. Let $\delta=\tau^{-1}$. 
Let $\Bbb N_0$ be the set of non-negative integers and let 
$\Bbb N_0^*=\{k/2: k\in \Bbb N_0\}$.  
Let $\eta'+\eta''=\eta, \eta'\in \widetilde{J_1}, 
\eta''\in \widetilde{J_2}$ and $\ell \in \Bbb N_0^*$, 
$\ell \leq \tau^{1/2}(\pi/8) -1$.  
Put 
\begin{equation*} 
p_{\ell, \delta}(\eta',\eta'')=
(\eta\cos(\ell\delta^{1/2}), (\eta'-\eta'')\sin(\ell\delta^{1/2})). 
\end{equation*} 
Let $R_0(a,b)=[-a, a]\times [-b, b]$, $a, b>0$ be the rectangle centered at 
$0$.  Let $\ell_*=\max(\ell, 1/2)$ and $c_1, c_2>0$. 
Define 
\begin{equation*} 
R(p_{\ell, \delta}(\eta',\eta''); c_1\ell_*\delta, c_2\delta^{1/2}) 
=p_{\ell, \delta}(\eta',\eta'') + R_0(c_1\ell_*\delta, c_2\delta^{1/2}).  
\end{equation*} 
\par 
Let $a\in \mathcal I$, $\ell\in \Bbb N_0^*$. 
If $\mu=a+\ell$, $\nu=a-\ell$, then 
\begin{equation}\label{e.3.1+}
\mathcal R_{-a\delta^{1/2}}(u_{\mu,J_1}^{\eta'}+u_{\nu,J_2}^{\eta''}) 
\subset  
R(p_{\ell, \delta}(\eta',\eta''); c_1\ell_*\delta, c_2\delta^{1/2})  
\end{equation}  
for some positive constants $c_1, c_2$, which can be seen since 
$u_{\mu,J_1}^{\eta'}+u_{\nu,J_2}^{\eta''}$  is contained in a ball of radius 
$c\delta^{1/2}$ and we have estimates similar to \eqref{e.2.3}.  
\par 
Let 
\begin{equation*} 
\mathscr E^{(0)}(\eta',\eta'')=\left\{(\xi_1, \xi_2): \frac{\xi_1^2}{\eta^2}+ 
\frac{\xi_2^2}{(\eta'-\eta'')^2}=1, \quad 1\leq \xi_1\leq \eta \right\}  
\end{equation*} 
when $\eta'\neq \eta''$; if $\eta'= \eta''$, let 
\begin{equation*} 
\mathscr E^{(0)}(\eta',\eta'')=\left\{(\xi_1, 0): 
1\leq \xi_1\leq \eta \right\}.
\end{equation*} 
Let 
\begin{equation*} 
\mathscr E(\eta',\eta'')=\left\{(\xi_1, \xi_2)\in 
\mathscr E^{(0)}(\eta',\eta''):  \xi_2\geq 0\right\} 
\end{equation*}
if $\eta'\geq \eta''$; when $\eta'\leq \eta''$ let 
\begin{equation*} 
\mathscr E(\eta',\eta'')=\left\{(\xi_1, \xi_2)\in 
\mathscr E^{(0)}(\eta',\eta''):  \xi_2\leq 0\right\}.  
\end{equation*} 
We note that the point 
$p_{\ell, \delta}(\eta',\eta'')$ is on the curve $\mathscr E(\eta',\eta'')$. 
Let $\mathscr E_a(\eta',\eta'')=
\mathcal R_{a\delta^{1/2}} \mathscr E(\eta',\eta'')$, $a\in \Bbb R$.  
\par 
Also, for a technical reason, it is convenient to consider a slightly 
augmented version of $\mathscr E(\eta',\eta'')$: 
\begin{equation*} 
\widetilde{\mathscr E}_{a}(\eta',\eta'')= 
\mathcal R_{a\delta^{1/2}}\widetilde{\mathscr E}(\eta',\eta''), \quad 
\widetilde{\mathscr E}(\eta',\eta'')=\mathscr E(\eta',\eta'')\cup 
\{(\xi_1,0): \eta\leq \xi_1\leq 5\}.  
\end{equation*}  
\par 
Let $A(\alpha,\beta)=\{\xi\in \Bbb R^2: \alpha\leq |\xi|\leq \beta\}$ be an 
annulus. 
To prove  Lemma \ref{lem1.6} we need the following. 

\begin{lemma} \label{lem3.1} 
Let $\ell \in \Bbb N_0^*\cap [0, \tau^{1/2}\pi 8^{-1}-1]$.  
Let $b_1$ be a positive constant satisfying 
$|p_{\ell, \delta}|-b_1\ell_* \delta> |\eta'-\eta''|$, where 
$p_{\ell, \delta}=p_{\ell, \delta}(\eta',\eta'')$.  
Then there exist $b_2, b_3>0$ depending on $b_1$ such that 
\begin{equation*} 
A(|p_{\ell, \delta}|-b_1\ell_* \delta, |p_{\ell, \delta}|+b_1\ell_* \delta)
\cap \widetilde{\mathscr E}(\eta',\eta'') \subset  
R(p_{\ell, \delta}; b_2\ell_*\delta, b_3\delta^{1/2}),  
\end{equation*} 
where $b_2$ and $b_3$ are independent of 
$\eta'\in \widetilde{J_1}$, $\eta''\in \widetilde{J_2}$ and $\delta$.   
\end{lemma}
\begin{proof} Let $\eta'\geq \eta''$. Let 
\begin{equation*} 
\Phi(\xi_1)=(\eta'-\eta'')\sqrt{1-\frac{\xi_1^2}{\eta^2}}, \quad 
\Psi(\xi_1)=\sqrt{(h-b_1\ell_*\delta)^2-\xi_1^2}, 
\end{equation*} 
where $h=|p_{\ell, \delta}|$.  If $\Phi(\xi_1)=\Psi(\xi_1)$, $\xi_1\geq 0$, 
then 
\begin{equation*} 
\xi_1= \frac{\eta}{\sqrt{\eta^2-\beta^2}}\sqrt{(h-b_1\ell_*\delta)^2-\beta^2}, 
\end{equation*} 
where $\beta=\eta'-\eta''$. We note that  
\begin{equation} \label{e.3.1}
0\leq \eta^2\cos^2(\ell\delta^{1/2})- \frac{\eta^2}{\eta^2-\beta^2}
((h-b_1\ell_*\delta)^2-\beta^2) 
\leq \frac{\eta^2}{\eta^2-\beta^2}2b_1h\ell_*\delta, 
\end{equation} 
as follows: 
\begin{align*} 
\eta^2\cos^2(\ell\delta^{1/2})- \frac{\eta^2}{\eta^2-\beta^2}
((h-b_1\ell_*\delta)^2-\beta^2)
&= \frac{\eta^2}{\eta^2-\beta^2}(h^2- (h-b_1\ell_*\delta)^2) 
\\ 
&\leq \frac{\eta^2}{\eta^2-\beta^2}2hb_1\ell_*\delta.
\end{align*} 
Since $\eta\cos(\ell\delta^{1/2})\geq 1$,   by \eqref{e.3.1} we have  
\begin{align} \label{e.3.2}
0&\leq \eta\cos(\ell\delta^{1/2})- \frac{\eta}{\sqrt{\eta^2-\beta^2}}
\sqrt{(h-b_1\ell_*\delta)^2-\beta^2} 
\\ 
&\leq 
\eta^2\cos^2(\ell\delta^{1/2})- \frac{\eta^2}{\eta^2-\beta^2}
((h-b_1\ell_*\delta)^2-\beta^2)                            \notag 
\\ 
&\leq \frac{\eta^2}{\eta^2-\beta^2}2b_1h\ell_*\delta.      \notag 
\end{align} 
\par 
In the case $\ell=0$, from \eqref{e.3.2} we can easily see that 
\begin{equation*}  
A(\eta-(b_1/2)\delta, \eta+(b_1/2) \delta)
\cap \widetilde{\mathscr E}(\eta',\eta'') \subset  
R((\eta,0); b_2\delta, b_3\delta^{1/2}) 
\end{equation*} 
for some $b_2, b_3>0$, which is what we need. So, we assume that $\ell\geq 1/2$ and $\ell_*=\ell$ in what follows. 
\par 
Let $\Phi(\xi_1)$ be as above and 
\begin{equation*} 
\widetilde{\Psi}(\xi_1)=\sqrt{(h+b_1\ell\delta)^2-\xi_1^2}.  
\end{equation*} 
Solving the equation $\Phi(\xi_1)=\widetilde{\Psi}(\xi_1)$ for $\xi_1\geq 0$ 
under the condition that $h+b_1\ell\delta \leq \eta$, 
we have 
\begin{equation*} 
\xi_1= \frac{\eta}{\sqrt{\eta^2-\beta^2}}\sqrt{(h+b_1\ell\delta)^2-\beta^2}.  
\end{equation*} 
We see that 
\begin{align*} 
0\leq \frac{\eta^2}{\eta^2-\beta^2}((h+b_1\ell\delta)^2-\beta^2)- 
\eta^2\cos^2(\ell\delta^{1/2})
&= \frac{\eta^2}{\eta^2-\beta^2}((h+b_1\ell\delta)^2-h^2) 
\\ 
&= \frac{\eta^2}{\eta^2-\beta^2}(2hb_1 +b_1^2\ell\delta)\ell\delta  
\end{align*} 
and hence, arguing as in \eqref{e.3.2}, we have 
\begin{align} \label{e.3.3}
0&\leq \frac{\eta}{\sqrt{\eta^2-\beta^2}}\sqrt{(h+b_1\ell\delta)^2-\beta^2} 
 -\eta\cos(\ell\delta^{1/2})  
\\ 
&\leq \frac{\eta^2}{\eta^2-\beta^2}((h+b_1\ell\delta)^2-\beta^2)- 
\eta^2\cos^2(\ell\delta^{1/2})                                       \notag 
\\ 
&=\frac{\eta^2}{\eta^2-\beta^2}(2hb_1 +b_1^2\ell\delta)\ell\delta,    \notag 
\end{align} 
assuming $h+b_1\ell\delta \leq \eta$. 
\par 
Next, we estimate 
\begin{equation*} 
I:=\Phi(\eta\cos(\ell\delta^{1/2})-\widetilde{b}_1\ell\delta)- 
\Phi(\eta\cos(\ell\delta^{1/2})), 
\end{equation*} 
where we assume that 
$\widetilde{b}_1\ell\delta\leq \eta\cos(\ell\delta^{1/2})$, 
$\widetilde{b}_1\geq 0$, and 
\begin{equation*} 
II:=-\Phi(\eta\cos(\ell\delta^{1/2})+\widetilde{b}_1\ell\delta)+
\Phi(\eta\cos(\ell\delta^{1/2})),   
\end{equation*} 
when  
 $\eta\cos(\ell\delta^{1/2})+\widetilde{b}_1\ell\delta\leq \eta$; 
when  $\eta\cos(\ell\delta^{1/2})+\widetilde{b}_1\ell\delta> \eta$, 
let $II=\Phi(\eta\cos(\ell\delta^{1/2}))$. 
We note that if $\eta< \eta\cos(\ell\delta^{1/2})+\widetilde{b}_1\ell\delta$, 
then 
\begin{equation*} 
\widetilde{b}_1\ell\delta>\eta(1-\cos(\ell\delta^{1/2}))
\geq \eta(2/\pi^2)\ell^2\delta, 
\end{equation*} 
and hence $\ell<\widetilde{b}_1\eta^{-1}\pi^2/2$.  
\par 
We use 
\begin{equation*} 
\Phi'(\xi_1)=\beta\left(1-\frac{\xi_1^2}{\eta^2}\right)^{-1/2}
\left(-\frac{\xi_1}{\eta^2} \right).  
\end{equation*} 
By the mean value theorem, it follows that 
\begin{equation} \label{e.3.4} 
|I|\leq \widetilde{b}_1\ell\delta\beta(1-\cos^2(\ell\delta^{1/2}))^{-1/2}
\eta^{-1}
\leq \widetilde{b}_1\ell\delta\beta \sin(\ell\delta^{1/2})^{-1} 
\leq (\pi/2)\widetilde{b}_1\beta\delta^{1/2}.  
\end{equation}  
\par   
Obviously, we see that  
\begin{equation} \label{e.3.5}
|II| \leq \Phi(\eta\cos(\ell\delta^{1/2}))=\beta\sin(\ell\delta^{1/2})\leq 
\beta\ell\delta^{1/2}. 
\end{equation} 
If $\ell\geq \widetilde{b}_1\eta^{-1}\pi^2/2$ and so 
$\eta\cos(\ell\delta^{1/2})+\widetilde{b}_1\ell\delta\leq \eta$, then 
applying the mean value theorem, we see that  
\begin{align*} 
|II|&\leq \beta\widetilde{b}_1\ell\delta\left(1-\eta^{-2}(\eta\cos(\ell\delta^{1/2})
+\widetilde{b}_1\ell\delta)^2\right)^{-1/2} 
\\ 
&= \beta\widetilde{b}_1\ell\delta\left(1-(\cos^2(\ell\delta^{1/2})
+2\eta^{-1}\widetilde{b}_1\ell\delta\cos(\ell\delta^{1/2}) 
+\eta^{-2}(\widetilde{b}_1)^2(\ell\delta)^2)\right)^{-1/2} 
\\ 
&= \beta\widetilde{b}_1\ell\delta\left(\sin^2(\ell\delta^{1/2})
-2\eta^{-1}\widetilde{b}_1\ell\delta\cos(\ell\delta^{1/2}) 
-\eta^{-2}(\widetilde{b}_1)^2(\ell\delta)^2\right)^{-1/2} 
\\ 
&\leq \beta\widetilde{b}_1\ell\delta\left((2/\pi)^2\ell^2\delta 
-(2\eta^{-1}\widetilde{b}_1  
+\eta^{-2}(\widetilde{b}_1)^2)\ell\delta\right)^{-1/2}=:J,   
\end{align*} 
where we assume that $\ell\geq 2(\pi/2)^2(2\eta^{-1}\widetilde{b}_1  
+\eta^{-2}(\widetilde{b}_1)^2)=:C_0$. 
Then,  we see that  
\begin{equation}\label{e.3.6} 
|II|\leq J\leq \beta\widetilde{b}_1\ell\delta\left(2^{-1}(2/\pi)^2\ell^2
\delta\right)^{-1/2}=2^{-1/2}\pi\beta\widetilde{b}_1\delta^{1/2}. 
\end{equation} 
By \eqref{e.3.5} and \eqref{e.3.6}, noting that 
$C_0 \geq \widetilde{b}_1\eta^{-1}\pi^2/2$,  we have 
\begin{equation}\label{e.3.7} 
|II| \leq 
\beta(C_0 + 2^{-1/2}\pi\widetilde{b}_1)\delta^{1/2}. 
\end{equation} 
\par 
By \eqref{e.3.2}, \eqref{e.3.3}, \eqref{e.3.4} and \eqref{e.3.7}, we can 
prove Lemma \ref{lem3.1} as follows. 
First, by \eqref{e.3.2}, \eqref{e.3.3}, we have 
\begin{multline}\label{e.3.0}  
A(|p_{\ell, \delta}|-b_1\ell \delta, |p_{\ell, \delta}|+b_1\ell \delta)
\cap \widetilde{\mathscr E}(\eta',\eta'') 
\\ 
\subset  [\eta\cos(\ell\delta^{1/2})- b_2\ell\delta, 
\eta\cos(\ell\delta^{1/2})+ b_2\ell\delta]\times \Bbb R 
\end{multline} 
for some $b_2>0$ under the condition  $h+b_1\ell\delta \leq \eta$. 
If $h+b_1\ell\delta > \eta$, we easily see that 
\begin{equation*} 
2hb_1\ell\delta +b_1^2\ell^2\delta^2>(\eta^2-\beta^2)\sin^2(\ell\delta^{1/2}), 
\end{equation*}  
which implies that $\ell\leq C$ for some constant $C$. Using this, we have 
\begin{align}\label{e.3.-1} 
|h-\eta\cos(\ell\delta^{1/2})| &\leq h^2-\eta^2\cos^2(\ell\delta^{1/2}) 
\\ 
&=(\eta'-\eta'')^2\sin^2(\ell\delta^{1/2})
\leq C_1\ell^2\delta \leq C_1C\ell\delta.       \notag 
\end{align}  
Also, when $h+b_1\ell\delta > \eta$, by \eqref{e.3.2} 
we see that 
\begin{equation}\label{e.3.-2}  
A(|p_{\ell, \delta}|-b_1\ell \delta, |p_{\ell, \delta}|+b_1\ell \delta)
\cap \widetilde{\mathscr E}(\eta',\eta'')  
\subset  [\eta\cos(\ell\delta^{1/2})- b_2\ell\delta, 
h+b_1\ell\delta]\times \Bbb R.  
\end{equation}  
By \eqref{e.3.-1} and \eqref{e.3.-2}, we also have \eqref{e.3.0} for some $b_2$ when $h+b_1\ell\delta > \eta$.  
\par 
Next, by \eqref{e.3.4} and \eqref{e.3.7} with 
$\widetilde{b}_1=b_2$ and \eqref{e.3.0} we have 
\begin{align*} 
&A(|p_{\ell, \delta}|-b_1\ell \delta, |p_{\ell, \delta}|+b_1\ell \delta)
\cap \widetilde{\mathscr E}(\eta',\eta'') 
\\
&\subset \left([\eta\cos(\ell\delta^{1/2})- b_2\ell\delta, 
\eta\cos(\ell\delta^{1/2})+ b_2\ell\delta]
\times \Bbb R\right)\cap \widetilde{\mathscr E}(\eta',\eta'') 
\\ 
&\subset  [\eta\cos(\ell\delta^{1/2})- b_2\ell\delta, 
\eta\cos(\ell\delta^{1/2})+ b_2\ell\delta] 
\\ 
&\phantom{\subset} \times 
[\beta\sin(\ell\delta^{1/2})-b_3\delta^{1/2}, 
\beta\sin(\ell\delta^{1/2})+b_3\delta^{1/2}] 
\end{align*} 
for some positive constant $b_3$.  
This proves Lemma \ref{lem3.1} when  $\eta'\geq \eta''$. 
\par 
The case $\eta'\leq \eta''$ can be handled similarly. 
This completes the proof of Lemma \ref{lem3.1}. 
\end{proof}  
\par 
We also need the following lemmas (Lemmas \ref{lem3.2}, \ref{lem3.3+} and 
\ref{lem3.4}) in proving Lemma \ref{lem1.6}.  

\begin{lemma} \label{lem3.2}  
Let $\ell \in \Bbb N_0^*\cap [0,\delta^{-1/2}(\pi/8)-1]$. 
Let $c_1$, $c_2$ be positive constants.
Then, there exist  constants 
$c_3, c_4>0$ depending on $c_1, c_2$ such that 
\begin{equation*} 
R(p_{\ell, \delta}; c_1\ell_*\delta, c_2\delta^{1/2}) \subset 
A(|p_{\ell, \delta}|-c_3\ell_*\delta, |p_{\ell, \delta}|+c_4\ell_*\delta),  
\end{equation*} 
where $p_{\ell, \delta}=p_{\ell, \delta}(\eta', \eta'')$ and 
$\ell_*=\max(\ell,1/2)$.  
\end{lemma} 
\begin{proof}
We write $(\alpha,\beta)$ for $p_{\ell,\delta}$.   Let 
$(\alpha+\epsilon_1,\beta+\epsilon_2)\in 
R(p_{\ell, \delta}; c_1\ell_*\delta, c_2\delta^{1/2})$. Then 
$|\epsilon_1|\leq c_1\ell_* \delta$, $|\epsilon_2|\leq c_2 \delta^{1/2}$.  
To prove the lemma, it suffices to show that 
\begin{equation*} 
\left|\sqrt{\alpha^2+\beta^2} - 
\sqrt{(\alpha+\epsilon_1)^2+(\beta+\epsilon_2)^2}\right|\leq c_0\ell_*\delta 
\end{equation*} 
for some $c_0>0$.  Since $\alpha\geq c>0$, this follows from the estimate 
\begin{equation} \label{e.3.8} 
\left|(\alpha^2+\beta^2) - 
\left((\alpha+\epsilon_1)^2+(\beta+\epsilon_2)^2\right)\right|
\leq c_0'\ell_*\delta.  
\end{equation} 
Now we see that $|\alpha|\leq 5$, $|\beta|\leq 
|\eta'-\eta''|\ell \delta^{1/2}$, $|\eta'-\eta''|\leq 3/2$ and 
\begin{align*} 
&\left|(\alpha^2+\beta^2) - 
\left((\alpha+\epsilon_1)^2+(\beta+\epsilon_2)^2\right)\right|
=\left|2\alpha\epsilon_1+\epsilon_1^2+2\beta\epsilon_2+\epsilon_2^2\right| 
\\ 
&\leq 10c_1\ell_*\delta +(c_1\ell_*\delta)^2
+ 2|\eta'-\eta''|\ell\delta^{1/2}c_2\delta^{1/2} +c_2^2\delta 
\\ 
&\leq 10c_1\ell_*\delta +(c_1\ell_*\delta)^2
+ 3c_2\ell_*\delta +c_2^2\delta 
\\ 
&\leq (10c_1 +c_1^2+ 3c_2 +2c_2^2)\ell_*\delta.  
\end{align*}  
This proves \eqref{e.3.8} and hence completes the proof of 
Lemma \ref{lem3.2}. 
\end{proof}  
\par 
Let $\eta' \in \widetilde{J_1}$, 
$\eta''\in \widetilde{J_2}$,  
$\eta'+\eta''=\eta$.    
Let $\ell \in \Bbb N_0^*$ and  
\begin{equation*}  
p_{\ell, \delta}(\eta',\eta'')= 
(\eta\cos(\ell\delta^{1/2}), (\eta'-\eta'')\sin(\ell\delta^{1/2})).  
\end{equation*} 
Let  
\begin{equation*} 
R_{\ell, \delta}(\eta',\eta'') 
=p_{\ell, \delta}(\eta',\eta'')  + 
R_0(c_1\ell_*\delta, (c_2+4)\delta^{1/2}),    
\end{equation*} 
where $c_1, c_2$ are as in \eqref{e.3.1+} and we recall that 
$\ell_*=\max(\ell,1/2)$.  

\begin{lemma} \label{lem3.3+} 
If $\eta', \eta_0'\in \widetilde{J_1}$, $\eta'', \eta_0''\in \widetilde{J_2}$, 
$\eta'+\eta''=\eta$, $\eta_0'+\eta_0''=\eta$, 
then 
\begin{equation*} 
\widetilde{\mathscr E}_a(\eta',\eta'')\cap R_{\ell,\delta}(\eta_0',\eta_0'') 
\subset 
\mathcal R_{a\delta^{1/2}} R(p_{\ell, \delta}(\eta',\eta''); 
c_1'\ell_*\delta, c_2'\delta^{1/2})   
\end{equation*} 
for some positive constants $c_1', c_2'$ independent of $a$, $\delta$ and 
$\ell$, where $\widetilde{\mathscr E}_a(\eta',\eta'')
=\mathcal R_{a\delta^{1/2}} \widetilde{\mathscr E}(\eta',\eta'')$ 
for $a\in \Bbb R$.    
\end{lemma} 
We have 
\begin{equation}\label{e.3.13++}  
R_{\ell,\delta}(\eta',\eta'')\cup  
R(p_{\ell, \delta}(\eta',\eta''); c_1'\ell_*\delta, c_2'\delta^{1/2}) 
\subset B(\eta, \ell, \delta, c_3) 
\end{equation}  
for all $\eta'\in \widetilde{J_1}$, $\eta''\in \widetilde{J_2}$ satisfying 
$\eta'+\eta''=\eta$ with some positive number $c_3$, where 
\begin{equation*} 
B(\eta, \ell, \delta, c_3)=\bigcup_{\substack{\eta'+\eta''=\eta, \\ 
\eta'\in \widetilde{J_1}, 
\eta''\in \widetilde{J_2}}} 
B(p_{\ell,\delta}(\eta',\eta''), c_3\delta^{1/2}).  
\end{equation*} 
Here $B(x,r)$ denotes a ball with radius $r$ centered at $x$. 
We may assume that $c_3\delta^{1/2}$ is small enough so that 
\begin{enumerate} 
\item 
\begin{equation*} 
\mathcal R_{\sigma}B(\eta, \ell, \delta, c_3)\subset 
D=\{\xi\in \Bbb R^2: 3/2\leq |\xi|\leq 9/2, \xi_1\geq 0\} 
\end{equation*} 
for $|\sigma|\leq \pi/4$; 
\item there exists $a_0>0$ independent of $\eta$, $\ell$, $\delta$ 
such that if $a_0 \leq |a|\leq (\pi/4)\delta^{-1/2}$, then 
\begin{equation} \label{e.3.13+}  
B(\eta, \ell, \delta, c_3) \cap 
\mathcal R_{a\delta^{1/2}}B(\eta, \ell, \delta, c_3)=\emptyset.  
\end{equation} 
\end{enumerate}  

\begin{proof}[Proof of Lemma $\ref{lem3.3+}$]  
We note that  
$|\eta'-\eta_0'|<2\delta^{1/2}$, $|\eta''-\eta_0''|<2\delta^{1/2}$. 
Thus 
\begin{equation*} 
  R_{\ell,\delta}(\eta_0',\eta_0'') \subset 
 R(p_{\ell, \delta}(\eta',\eta''); 
c_1\ell_*\delta, (c_2+8)\delta^{1/2}).   
\end{equation*} 
So, by Lemma \ref{lem3.2} we have 
\begin{equation*} 
R_{\ell,\delta}(\eta_0',\eta_0'')  \subset  R(p_{\ell, \delta}; 
c_1\ell_*\delta, (c_2+8)\delta^{1/2}) 
\subset 
A(|p_{\ell, \delta}|-c\ell_*\delta, |p_{\ell, \delta}|+c\ell_*\delta)  
\end{equation*} 
for some $c>0$, where $p_{\ell, \delta}=p_{\ell, \delta}(\eta',\eta'')$.  Thus 
\begin{equation*} 
\widetilde{\mathscr E}_a(\eta',\eta'')\cap 
R_{\ell,\delta}(\eta_0',\eta_0'') \subset 
\widetilde{\mathscr E}_a(\eta',\eta'')\cap 
A(|p_{\ell, \delta}|-c\ell_*\delta, |p_{\ell, \delta}|+c\ell_*\delta).   
\end{equation*} 
By Lemma \ref{lem3.1} and the rotation invariance of annulus, we see that 
\begin{equation*} 
\widetilde{\mathscr E}_a(\eta',\eta'')\cap 
A(|p_{\ell, \delta}|-c\ell_* \delta, |p_{\ell, \delta}|+c\ell_* \delta)
 \subset  
\mathcal R_{a\delta^{1/2}} R(p_{\ell, \delta}; c'\ell_*\delta, c'\delta^{1/2}) 
\end{equation*} 
for some $c'>0$. Combining results, we have 
\begin{equation*} 
\widetilde{\mathscr E}_a(\eta',\eta'')\cap  
R_{\ell,\delta}(\eta_0',\eta_0'')  \subset 
\mathcal R_{a\delta^{1/2}} R(p_{\ell, \delta}; 
c'\ell_*\delta, c'\delta^{1/2}).  
\end{equation*} 
This completes the proof of Lemma \ref{lem3.3+}. 
\end{proof}   
\par 
Let 
\begin{equation*} 
\widetilde{E}^\eta_{00}=\bigcup_{\substack{\ell, \eta'+\eta''=\eta, \\ 
\eta'\in \widetilde{J_1}, 
\eta''\in \widetilde{J_2}}} R(p_{\ell,\delta}(\eta',\eta''); 
c_1\ell_*\delta, (c_2+4)\delta^{1/2})
=\bigcup_{\substack{\ell, \eta'+\eta''=\eta, \\ 
\eta'\in \widetilde{J_1}, 
\eta''\in \widetilde{J_2}}} R_{\ell,\delta}(\eta_,\eta''), 
\end{equation*} 
where $\ell$ ranges over a subset of $\Bbb N_0^*$ such that 
$0\leq \ell \leq \delta^{-1/2}(\pi/8)-1$ and $c_1, c_2$ are as in 
\eqref{e.3.1+}.

\begin{lemma} \label{lem3.4}  
Fix $\eta'\in \widetilde{J_1}$ and $\eta''\in \widetilde{J_2}$ with 
$\eta'+\eta''=\eta$. 
There exists $a_0>0$ independent of $\delta$ such that if 
$a_0\leq a \leq (\pi/8)\delta^{-1/2}$, then 
\begin{equation*} 
\widetilde{\mathscr E}_a(\eta',\eta'')\cap \widetilde{E}^\eta_{00}  
=\emptyset, \quad 
\mathcal R_{a\delta^{1/2}}(\widetilde{E}^\eta_{00})
\cap \widetilde{\mathscr E}(\eta',\eta'')=\emptyset.  
\end{equation*} 
\end{lemma}

\begin{proof} 
Let $\eta_0'\in \widetilde{J_1}$ and $\eta_0''\in \widetilde{J_2}$ with 
$\eta_0'+\eta_0''=\eta$. 
We first show that   
\begin{equation} \label{e.3.9} 
\widetilde{\mathscr E}_a(\eta',\eta'')\cap R_{\ell,\delta}(\eta_0',\eta_0'') 
=\emptyset,  
\end{equation} 
if $a_0\leq a \leq (\pi/8)\delta^{-1/2}$ and $a_0$ is sufficiently large, 
where $R_{\ell,\delta}(\eta_0',\eta_0'')$ is as in Lemma \ref{lem3.3+}.  
By Lemma \ref{lem3.3+} and \eqref{e.3.13++}, we have 
\begin{equation*} 
\widetilde{\mathscr E}_a(\eta',\eta'')\cap R_{\ell,\delta}(\eta_0',\eta_0'') 
\subset \mathcal R_{a\delta^{1/2}} B(\eta, \ell, \delta, c_3).    
\end{equation*} 
Since $R_{\ell,\delta}(\eta_0',\eta_0'') \subset  B(\eta, \ell, \delta, c_3)$, 
by \eqref{e.3.13+} we have \eqref{e.3.9} 
if $a_0\leq a \leq (\pi/8)\delta^{-1/2}$ and 
$a_0$ is as in \eqref{e.3.13+}, 
 from which it can be deduced that 
$\widetilde{\mathscr E}_a(\eta',\eta'')\cap \widetilde{E}^\eta_{00} 
=\emptyset$ as claimed. 
\par 
Next, we prove that 
$\mathcal R_{a\delta^{1/2}}(\widetilde{E}^\eta_{00})
\cap \widetilde{\mathscr E}(\eta',\eta'')=\emptyset$ if 
$a_0\leq a \leq (\pi/8)\delta^{-1/2}$ and $a_0$ is sufficiently large. 
This follows from $\widetilde{E}^\eta_{00}
\cap \widetilde{\mathscr E}_{-a}(\eta',\eta'')
=\emptyset$, which can be shown as above by using Lemma \ref{lem3.3+}, 
\eqref{e.3.13++} and \eqref{e.3.13+}. 
This completes the proof of Lemma \ref{lem3.4}. 
\end{proof}  

\begin{proof}[Proof of Lemma $\ref{lem1.6}$]  
Let 
\begin{equation*} 
\widetilde{E}^\eta_{0}=\bigcup_{\substack{\ell, \eta'+\eta''=\eta, \\ 
\eta'\in \widetilde{J_1}, 
\eta''\in \widetilde{J_2}}} R(p_{\ell,\delta}(\eta',\eta''); 
c_1\ell_*\delta, c_2\delta^{1/2}), 
\end{equation*} 
where the constants $c_1, c_2$ are as in \eqref{e.3.1+} and 
$\ell$ ranges over a subset of $\Bbb N_0^*$ such that 
$0\leq \ell \leq \delta^{-1/2}(\pi/8)-1$.  
We note that 
\begin{equation} \label{e.3.11} 
 E_\alpha^\eta \subset \mathcal R_{\alpha\delta^{1/2}} \widetilde{E}^\eta_0   
=\bigcup_{\substack{\ell, \eta'+\eta''=\eta, \\ 
\eta'\in \widetilde{J_1}, 
\eta''\in \widetilde{J_2}}} 
\mathcal R_{\alpha\delta^{1/2}} R(p_{\ell,\delta}(\eta',\eta''); 
c_1\ell_*\delta, c_2\delta^{1/2})   
\end{equation} 
(see \eqref{e.3.1+}). 
Recall that 
$D=\{\xi\in \Bbb R^2: 3/2\leq |\xi|\leq 9/2, \xi_1\geq 0\}$. 
We may assume that $\mathcal R_{\alpha\delta^{1/2}} \widetilde{E}^\eta_0 
\subset D$ for $|\alpha|\leq (\pi/4)\delta^{-1/2}$, $\eta=\eta'+\eta'', 
\eta'\in \widetilde{J_1}, \eta''\in \widetilde{J_2}$. 
\par 
Let $\ell, \ell' \in \Bbb N_0^*$ with 
$0\leq \ell, \ell'\leq \delta^{-1/2}(\pi/8)-1$, 
$\eta_0', \eta_1'\in \widetilde{J_1}$, $\eta_0'', \eta_1''\in \widetilde{J_2}$ 
with $\eta_0'+\eta_0''=\eta$,  $\eta_1'+\eta_1''=\eta$. 
Let $a, a'\in \mathcal I$. 
To prove Lemma \ref{lem1.6}, by \eqref{e.3.11} it suffices  to show that 
\begin{equation} \label{e.3.12} 
R(p_{\ell,\delta}(\eta_0',\eta_0''), c_1\ell_*\delta, c_2\delta^{1/2}) 
\cap \mathcal R_{(a-a')\delta^{1/2}}R(p_{\ell',\delta}(\eta_1',\eta_1''), 
c_1\ell_*'\delta, c_2\delta^{1/2})= \emptyset,   
\end{equation} 
if $a-a'$ is sufficiently large with $a-a'\leq (\pi/4)\delta^{-1/2}-2$. 
Let $b=a-a'$.  
We observe that for $\eta'\in \widetilde{J_1}$ and 
$\eta''\in \widetilde{J_2}$ with $\eta'+\eta''=\eta$, 
\begin{gather} 
p_{\ell,\delta}(\eta',\eta'') \in R(p_{\ell,\delta}(\eta_0',\eta_0''), 
c_1\ell_*\delta, (c_2+4)\delta^{1/2})=R_{\ell, \delta}(\eta_0',\eta_0''), 
\label{e.3.16+} 
\\ 
\mathcal R_{b\delta^{1/2}}p_{\ell',\delta}(\eta',\eta'')\in 
\mathcal R_{b\delta^{1/2}}R(p_{\ell',\delta}(\eta_1',\eta_1''), 
c_1\ell_*'\delta, (c_2+4)\delta^{1/2})
=\mathcal R_{b\delta^{1/2}}R_{\ell', \delta}(\eta_1',\eta_1'').    \notag 
\end{gather} 
Obviously \eqref{e.3.12} follows from 
\begin{equation} \label{e.3.17+} 
R_{\ell, \delta}(\eta_0',\eta_0'') \cap 
\mathcal R_{b\delta^{1/2}}R_{\ell', \delta}(\eta_1',\eta_1'')= \emptyset.  
\end{equation} 
\par 
Applying Lemma \ref{lem3.4}, we see that 
\begin{align}\label{e.3.18+} 
\widetilde{\mathscr E}_{b/2}(\eta',\eta'') \cap 
R_{\ell, \delta}(\eta_0',\eta_0'')
&=\emptyset, 
\\ 
\mathcal R_{b\delta^{1/2}}R_{\ell', \delta}(\eta_1',\eta_1'')
\cap \widetilde{\mathscr E}_{b/2}(\eta',\eta'') &=\emptyset \notag 
\end{align}
for a sufficiently large $b>0$, $0<b\leq (\pi/4)\delta^{-1/2}-2$.  
\par 
We can divide $D$ as 
$D\setminus \widetilde{\mathscr E}_{b/2}(\eta',\eta'')=D_1\cup D_2$ with  
$D_1\cap D_2=\emptyset$. 
Since $\mathscr E(\eta', \eta''), \mathscr E_b(\eta', \eta'') \subset 
D$ and, obviously, $\widetilde{\mathscr E}_{b/2}(\eta',\eta'')\cap 
\mathscr E(\eta', \eta'')=\emptyset$,  
$\widetilde{\mathscr E}_{b/2}(\eta',\eta'')\cap 
\mathscr E_b(\eta', \eta'')=\emptyset$, 
we may assume that $\mathscr E(\eta', \eta'') \subset D_1$ and 
 $\mathscr E_b(\eta', \eta'') \subset D_2$. 
Since 
$p_{\ell,\delta}(\eta',\eta'') \in \mathscr E(\eta',\eta'')$,  
 $\mathcal R_{b\delta^{1/2}}p_{\ell',\delta}(\eta',\eta'')\in 
\mathscr E_b(\eta',\eta'')$, we have  
$p_{\ell,\delta}(\eta',\eta'') \in D_1$ and  
 $\mathcal R_{b\delta^{1/2}}p_{\ell',\delta}(\eta',\eta'')\in D_2$.  
Thus by \eqref{e.3.16+} and \eqref{e.3.18+} we have 
$R_{\ell, \delta}(\eta_0',\eta_0'') \subset D_1$ 
and 
$\mathcal R_{b\delta^{1/2}}R_{\ell', \delta}(\eta_1',\eta_1'') 
\subset D_2$, 
which implies \eqref{e.3.17+}. 
This completes the proof of Lemma \ref{lem1.6}. 
\end{proof}   

\section{Applications of Theorem $\ref{th1.1}$}\label{s4}

Recall that 
\begin{align*}  
u_{\mu,J_i}=u_{\mu, J_i}^{(\delta)}&=\{(\xi,\eta)\in \Bbb R^2\times 
\Bbb R: |\eta-|\xi||\leq \delta, (\xi,|\xi|)\in U_{\mu,J_i}\} 
\\ 
&= \cup \{\{\xi\}\times [|\xi|-\delta, |\xi|+\delta]: 
(\xi,|\xi|)\in U_{\mu,J_i}\}  \quad i=1, 2,                                     \notag 
\end{align*} 
where $ \delta= \tau^{-1}$ is a small positive number, 
\begin{align*} 
U_{\mu,J_i}&=U_{\mu,J_i}^{(\delta)}=
\{(\xi,|\xi|)\in \Bbb R^2\times \Bbb R: \xi\in \Gamma_\mu, 
|\xi|\in J_i\}, 
\\ 
\Gamma_\mu&=\Gamma_\mu^{(\delta)}
=\{\xi \in \Bbb R^2: \mu\delta^{1/2}\leq \arg \xi <(\mu+1)\delta^{1/2}\},  
\quad \mu\in \mathcal N, 
\\ 
\mathcal N&=\mathcal N^{(\delta)}=\{\mu\in \Bbb Z: |\mu|\leq \frac{\pi}{8}
\delta^{-1/2} -1\},  
\\ 
J_i&=J_i^{(\delta)}=[\alpha_i, \beta_i]\subset [1, 2],  
\quad |J_i|\leq \delta^{1/2}, \quad i=1, 2. 
\end{align*} 
We also consider 
\begin{equation*} 
\mathcal N_*^{(\delta)}=\{\mu\in \Bbb Z: 
|\mu|\leq \frac{\pi}{7}\delta^{-1/2} -1\}. 
\end{equation*} 
For $\mu \in \mathcal N_*^{(\delta)}$, 
we define an enlargement $(u_{\mu, J_i}^{(\delta)})^*$ of 
$u_{\mu, J_i}^{(\delta)}$, i=1, 2,  by   
\begin{align*}  
(u_{\mu, J_i}^{(\delta)})^*&=\{(\xi,\eta)\in \Bbb R^2\times 
\Bbb R: |\eta-|\xi||\leq 6\delta, (\xi,|\xi|)\in (U_{\mu, J_i}^{(\delta)})^*\} 
\\ 
&= \cup \{\{\xi\}\times [|\xi|-6\delta, |\xi|+6\delta]: 
(\xi,|\xi|)\in (U_{\mu, J_i}^{(\delta)})^*\},                                  
\notag 
\end{align*} 
where 
\begin{align*} 
(U_{\mu, J_i}^{(\delta)})^*&=
\{(\xi,|\xi|)\in \Bbb R^2\times \Bbb R: \xi\in (\Gamma_\mu^{(\delta)})^*, 
|\xi|\in (J_i^{(\delta)})^* \}, 
\\ 
(\Gamma_\mu^{(\delta)})^* 
&=\{\xi \in \Bbb R^2: 
(\mu-1)\delta^{1/2}\leq \arg \xi <(\mu+2)\delta^{1/2}\},  
\\  
(J_i^{(\delta)})^*&=[\alpha_i-2\delta, \beta_i+2\delta]. 
\end{align*}
\par 
We have the following result by examining the proof of Theorem \ref{th1.1}. 

\begin{theorem} \label{th4.1}  There exists a positive constant $C$ 
independent of $J_1, J_2$ and $\delta$ such that 
\begin{equation*}  
\sum_{(\mu, \nu)\in (\mathcal N_*^{(\delta)})^2} 
\chi_{(u_{\mu,J_1}^{(\delta)})^*+(u_{\nu,J_2}^{(\delta)})^*}
\leq C.  
\end{equation*} 
\end{theorem} 
Let $0<\epsilon <1/2$ and set 
\begin{equation*} 
\widetilde{u}_{\mu,J_i}=\{(\xi, \eta)\in \Bbb R^2\times \Bbb R: 
d(u_{\mu,J_i}, (\xi,\eta)) <\delta^{1-\epsilon}\}, \quad i=1, 2,      
\end{equation*} 
where 
\begin{equation*} 
d(u_{\mu,J_i}, (\xi,\eta))=\inf_{(\xi',\eta')\in u_{\mu,J_i}} 
(|\xi_1-\xi'_1|^2+|\xi_2-\xi'_2|^2+|\eta-\eta'|^2)^{1/2}.  
\end{equation*} 
\par 
Then we have the following. 
\begin{theorem} \label{th4.2}  We can find a positive constant $C$ 
independent of $J_1, J_2$ and $\delta$ such that 
\begin{equation*}  
\sum_{(\mu, \nu)\in \mathcal N^2} 
\chi_{\widetilde{u}_{\mu,J_1}+\widetilde{u}_{\nu,J_2}}
\leq C\delta^{-\epsilon}.  
\end{equation*} 
\end{theorem} 
To prove Theorem \ref{th4.2} by applying Theorem \ref{th4.1}, we need the 
following. 

\begin{lemma} \label{lem4.3} 
Let $N=[\delta^{-\epsilon/2}]$, 
$\delta_\epsilon=\delta[\delta^{-\epsilon/2}]^2\sim \delta^{1-\epsilon}$, 
where $[\alpha]=\max\{m\in \Bbb Z: m\leq \alpha\}$ for $\alpha\in \Bbb R$.  
 Suppose that $\mu\in \mathcal N^{(\delta)}$ and $\mu=\ell N+k$ for some 
$\ell\in \Bbb Z$ and $k\in [0,N-1]\cap \Bbb Z$.  
Then 
\begin{equation*} 
\Gamma_{\ell N+k}^{(\delta)} \subset 
\{\ell \delta_\epsilon^{1/2} \leq \arg \xi< 
(\ell +1)\delta_\epsilon^{1/2}\}. 
\end{equation*} 
\end{lemma} 
\begin{proof} 
We have  
\begin{align*} 
\Gamma_{\ell N+k}^{(\delta)}&=\{(\ell N+k)\delta^{1/2} \leq \arg \xi< 
(\ell N+k+1)\delta^{1/2}\} 
\\ 
&\subset \{\ell N\delta^{1/2} \leq \arg \xi< 
(\ell N+N)\delta^{1/2}\} 
\\ 
&= \{\ell N\delta^{1/2} \leq \arg \xi< 
(\ell +1)N\delta^{1/2}\} 
\\ 
&= \{\ell \delta_\epsilon^{1/2} \leq \arg \xi< 
(\ell +1)\delta_\epsilon^{1/2}\}.  
\end{align*} 
This completes the proof. 
\end{proof} 
We also need the following. 
\begin{lemma} \label{lem4.4} 
Let $\mu, \ell, N, k$ be as in Lemma $\ref{lem4.3}$ with $\ell \in 
\mathcal N_*^{(\delta_\epsilon)}$. Then, 
if $\delta$ is small enough,  we have 
\begin{equation*}  
\widetilde{u}_{\mu, J_i}= \widetilde{u}_{\ell N+k, J_i} 
\subset (u_{\ell, J_i}^{(\delta_\epsilon)})^*, \quad i=1, 2.  
\end{equation*} 
\end{lemma} 
\begin{proof} 
We show the result by using Lemma \ref{lem4.3} and the definitions of 
$\widetilde{u}_{\mu, J_i}$, $u_{\ell, J_i}^{(\delta_\epsilon)}$ and 
$(u_{\ell, J_i}^{(\delta_\epsilon)})^*$ as follows.  
\par 
Fix $i$ and let $(\xi, \eta)\in \widetilde{u}_{\ell N+k, J_i}$.  
Then there exists 
$(\xi_0, \eta_0)\in u_{\ell N+k, J_i}$ such that 
\begin{equation} \label{e.4.1} 
|(\xi, \eta)-(\xi_0, \eta_0)|< \delta^{1-\epsilon}.  
\end{equation} 
Since $(\xi_0, \eta_0)\in u_{\ell N+k, J_i}$, we  
have $(\xi_0, |\xi_0|)\in U_{\ell N+k, J_i}$ and 
\begin{equation} \label{e.4.2} 
|\eta_0-|\xi_0||< \delta.  
\end{equation} 
The fact $(\xi_0, |\xi_0|)\in U_{\ell N+k, J_i}$  implies that 
$\xi_0\in \Gamma_{\ell N+k}^{(\delta)}$ and 
$|\xi_0|\in J_i$. 
By Lemma \ref{lem4.3} it follows that 
\begin{equation}\label{e.4.3} 
\ell\delta_\epsilon^{1/2}\leq \arg \xi_0 < (\ell+1)\delta_\epsilon^{1/2},   
\end{equation} 
where $|\ell|\leq (\pi/7)\delta_\epsilon^{-1/2}-1$ by the assumption that 
$\ell\in \mathcal N_*^{(\delta_\epsilon)}$.  
\par 
To prove $(\xi, \eta)\in (u_{\ell, J_i}^{(\delta_\epsilon)})^*$, we need 
the estimate 
\begin{equation} \label{e.4.5} 
|\arg \xi -\arg \xi_0|\leq \delta_\epsilon^{1/2}, 
\end{equation} 
if $0<\epsilon<1/2$ and $\delta$ is small enough. 
By \eqref{e.4.1}, \eqref{e.4.3} and the fact that $|\xi_0|\in J_i$, 
 we have $1/2<|\xi|<5/2$ and $|\arg \xi|<\pi/4$. Let 
$\xi_0=(\zeta_1, \zeta_2)$. 
Using \eqref{e.4.1}, if $\delta$ is small enough, we see that 
\begin{align*} 
|\arg\xi -\arg\xi_0|&=|\arctan(\xi_2/\xi_1)- \arctan(\zeta_2/\zeta_1)| 
\\ 
&\leq |\xi_2/\xi_1 -\zeta_2/\zeta_1| 
= |\xi_2\zeta_1-\xi_1\zeta_2|/|\xi_1\zeta_1|               \notag 
\\ 
&\leq c|\xi_1-\zeta_1|+ c|\xi_2-\zeta_2|\leq 2c\delta^{1-\epsilon} \notag 
\\ 
&\leq   4c\delta_\epsilon   \leq  \delta_\epsilon^{1/2},                                 \notag 
\end{align*} 
which proves \eqref{e.4.5}, 
where we have also used the estimates $\delta^{1-\epsilon}/2\leq 
\delta_\epsilon\leq \delta^{1-\epsilon}$, which are valid when 
$\delta$ is small enough.  
\par 
By \eqref{e.4.3} and \eqref{e.4.5} we have 
\begin{equation} \label{e.4.6} 
(\ell-1)\delta_\epsilon^{1/2}\leq \arg \xi < (\ell+2)\delta_\epsilon^{1/2}. 
\end{equation} 
Since $|\xi_0|\in J_i=[\alpha_i, \beta_i]$, by \eqref{e.4.1} it follows that 
\begin{equation}\label{e.4.7} 
\alpha_i-2\delta_\epsilon \leq |\xi|\leq \beta_i+2\delta_\epsilon. 
\end{equation} 
Also \eqref{e.4.1} and \eqref{e.4.2} imply that 
\begin{equation}\label{e.4.8} 
|\eta-|\xi||\leq |\eta-\eta_0|+|\eta_0-|\xi_0||+||\xi_0|-|\xi|| 
< \delta^{1-\epsilon}+ \delta + \delta^{1-\epsilon}<3\delta^{1-\epsilon} 
<6\delta_\epsilon. 
\end{equation} 
By \eqref{e.4.6} with $\ell\in \mathcal N_*^{(\delta_\epsilon)}$ 
 and \eqref{e.4.7} we see that $(\xi, |\xi|)\in 
(U_{\ell, J_i}^{(\delta_\epsilon)})^*$, which combined with \eqref{e.4.8} 
will imply that 
$(\xi, \eta)\in (u_{\ell, J_i}^{(\delta_\epsilon)})^*$. 
This completes the proof of Lemma \ref{lem4.4}. 
\end{proof}  
The assumption that $\ell \in \mathcal N_*^{(\delta_\epsilon)}$ in 
Lemma \ref{lem4.4} is always satisfied when $\delta$ is small. 
\begin{lemma}\label{lem4.5} 
If $\mu \in \mathcal N^{(\delta)}$, then there exist $\ell \in 
\mathcal N_*^{(\delta_\epsilon)}$ and $k\in \Bbb Z$ with $0\leq k 
\leq N-1$ such that $\mu=\ell N+k$, where $N$ is as in Lemma $\ref{lem4.3}$. 
\end{lemma} 
\begin{proof} 
We write $\mu=mN+k$, $m, k\in \Bbb Z$ with $0\leq k\leq N-1$. Then 
\begin{align*} 
|m|&= N^{-1}|\mu -k|\leq N^{-1}(\frac{\pi}{8}\delta^{-1/2}-1 +k) 
\leq N^{-1}(\frac{\pi}{8}\delta^{-1/2} +N-2) 
\\ 
&=\frac{\pi}{8}\delta_\epsilon^{-1/2}+1 -2/N
\leq \frac{\pi}{7}\delta_\epsilon^{-1/2}-1, 
\end{align*} 
if $\delta$ is small enough, which implies that 
$m \in \mathcal N_*^{(\delta_\epsilon)}$.  
\end{proof}

\begin{proof}[Proof of Theorem $\ref{th4.2}$] 
We may assume that $\delta$ is small enough. 
Let $\ell, \ell'\in \mathcal N_*^{(\delta_\epsilon)}$, 
$0\leq k, k'\leq N-1$ and $\ell N+k, \ell' N+k' \in 
\mathcal N^{(\delta)}$.  
Then 
by Lemma \ref{lem4.4} we have 
\begin{equation*} 
\chi_{\widetilde{u}_{\ell N+k, J_1}+
 \widetilde{u}_{\ell' N+k', J_2}} 
\leq \chi_{(u_{\ell, J_1}^{(\delta_\epsilon)})^*+
(u_{\ell', J_2}^{(\delta_\epsilon)})^*}. 
\end{equation*} 
Therefore applying Lemma \ref{lem4.5} and Theorem \ref{th4.1} with 
$\delta_\epsilon$ in place of $\delta$, we have 
\begin{align*}  
\sum_{(\mu, \nu)\in \mathcal N^2} 
\chi_{\widetilde{u}_{\mu,J_1}+\widetilde{u}_{\nu,J_2}} 
&\leq \sum_{\ell, \ell'\in \mathcal N_*^{(\delta_\epsilon)}}
\sum_{k, k'\in [0,N-1]\cap \Bbb Z} 
\chi_{\widetilde{u}_{\ell N+k, J_1}+
 \widetilde{u}_{\ell' N+k', J_2}} 
\\ 
&\leq N^2 \sum_{(\ell, \ell')\in (\mathcal N_*^{(\delta_\epsilon)})^2} 
\chi_{(u_{\ell, J_1}^{(\delta_\epsilon)})^*+ 
(u_{\ell', J_2}^{(\delta_\epsilon)})^*}  
\\ 
&\leq CN^2 
\\ 
&\leq C\delta^{-\epsilon}. 
\end{align*} 
This completes the proof of Theorem \ref{th4.2}.  
\end{proof}

\end{document}